\documentclass[11pt]{article}

\usepackage[margin=1in]{geometry}
\usepackage{amsthm}
\usepackage{amssymb}
\usepackage{amsmath}
\usepackage{graphicx}
\usepackage{url}
\usepackage{color}

\newcommand\blfootnote[1]{%
  \begingroup
  \renewcommand\thefootnote{}\footnote{#1}%
  \addtocounter{footnote}{-1}%
  \endgroup
}

\newtheorem{theorem}{Theorem}

\newtheorem{lemma}[theorem]{Lemma}

\theoremstyle{definition}

\title{Partisan gerrymandering with geographically compact districts}

\author{Boris Alexeev\qquad Dustin~G.~Mixon\footnote{Department of Mathematics, The Ohio State University, Columbus, OH}}
\date{}

\begin{document}
\maketitle

\blfootnote{Send correspondence to \texttt{mixon.23@osu.edu}}

\begin{abstract}
Bizarrely shaped voting districts are frequently lambasted as likely instances of gerrymandering.
In order to systematically identify such instances, researchers have devised several tests for so-called geographic compactness (i.e., shape niceness).
We demonstrate that under certain conditions, a party can gerrymander a competitive state into geographically compact districts to win an average of over 70\% of the districts.
Our results suggest that geometric features alone may fail to adequately combat partisan gerrymandering.
\end{abstract}

\section{Introduction}

Gerrymandering is the manipulation of voting district boundaries to obtain political advantage.
The term was coined in 1812 by the Boston Gazette, which likened the contorted shape of a Massachusetts voting district to the profile of a salamander.
Ever since, voting districts with bizarre shapes have been criticized as likely instances of gerrymandering (for example, see~\cite{Ingraham:14}).
The U.S.\ Supreme Court is currently deliberating over whether a non-geometric feature should be used to detect partisan gerrymandering.
Instead of using geometry to detect apparent boundary manipulation, the proposal uses recent election data to detect apparent political advantage; the proposal quantifies this advantage using the so-called efficiency gap~\cite{StephanopoulosM:15,BernsteinD:17}.

In this paper, we show that under a certain conditions, a party can gerrymander a competitive state into nicely shaped districts and still manage to win an average of over 70\% of the districts; see Figure~\ref{fig.example} for a real-world instance of this phenomenon.
This suggests that geometric features may be insufficient to adequately diagnose partisan gerrymandering, meaning additional non-geometric features such as efficiency gap may be necessary to do the job.

To formalize the notion of shape niceness, researchers have devised several methods to quantify so-called geographic compactness, and they can be roughly classified into three different types~\cite{Duchin:17}:
\begin{itemize}
\item
\textbf{Isoperimetry.}
Intuitively, a gerrymandered district will spend much of its perimeter selectively including and excluding various portions of a map.
One could quantify this waste by simply measuring the perimeter.
A scale-invariant alternative is the Polsby--Popper score, given by the ratio between the area of the district and the square of its perimeter~\cite{PolsbyR:91}.
\item
\textbf{Convexity.}
Congressional districts are confined to state borders, which often exhibit jagged portions due to geographic features such as rivers.
These features then induce long district perimeters, and so the perimeter fails to capture the geometric waste due to gerrymandering.
Alternatively, one may compare the area of the district to its convex hull, or to the area of the smallest disk containing the district (as in the Reock score~\cite{Roeck:61}).
\item
\textbf{Dispersion.}
Another common feature among gerrymandered districts is \textit{sprawl}.
This is quantified by computing the average distance between pairs of points in the district, or by computing the district's moment of inertia.
In particular, given a Lebesgue measurable district $D\subseteq\mathbb{R}^2$, the centroid and moment of inertia are given by
\[
\mu_D=\frac{1}{|D|}\iint_D x~d\operatorname{Leb}(x),
\qquad
I_D=\iint_D \|x-\mu_D\|^2d\operatorname{Leb}(x).
\]
\end{itemize}

In the following section, we introduce a model for voter locations and preferences.
In this setting, we show that using a line to split a circular state into two equal populations produces districts with optimal geographic compactness, in the sense of isoperimetry, convexity and dispersion simultaneously.
We then report our main result:
Under our model of voter locations and preferences, one may split a circular state into two such districts, winning an average of
\[
\frac{3}{4}-\frac{1}{2(1+e^\pi)}
\approx 73\%
\]
of the districts (in the limit as the number of voters goes to infinity).
The proof of this result involves passing from random walks to Brownian motion, and then manipulating instances of Brownian motion to compute exact probabilities.
Section~3 provides the proof of the main result, and Section~4 contains various technical lemmas.

\begin{figure}
\centering
\includegraphics[width=0.49\textwidth]{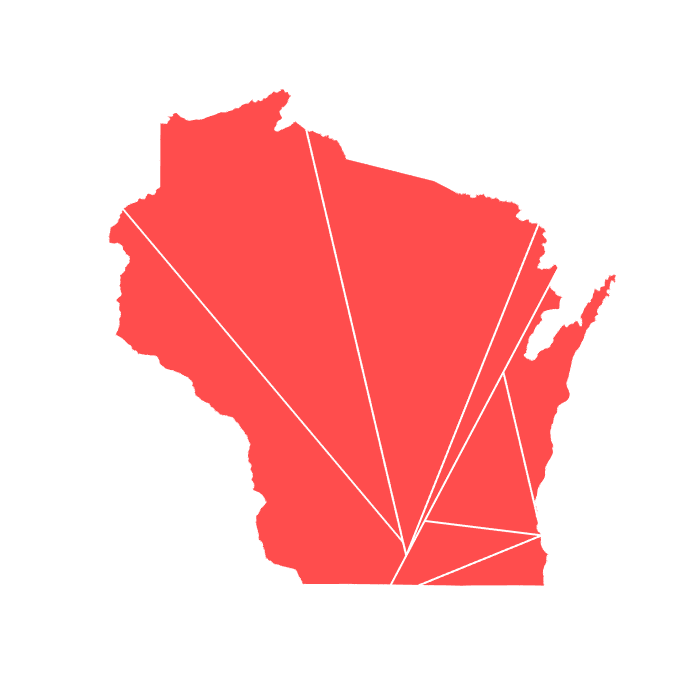}
\includegraphics[width=0.49\textwidth]{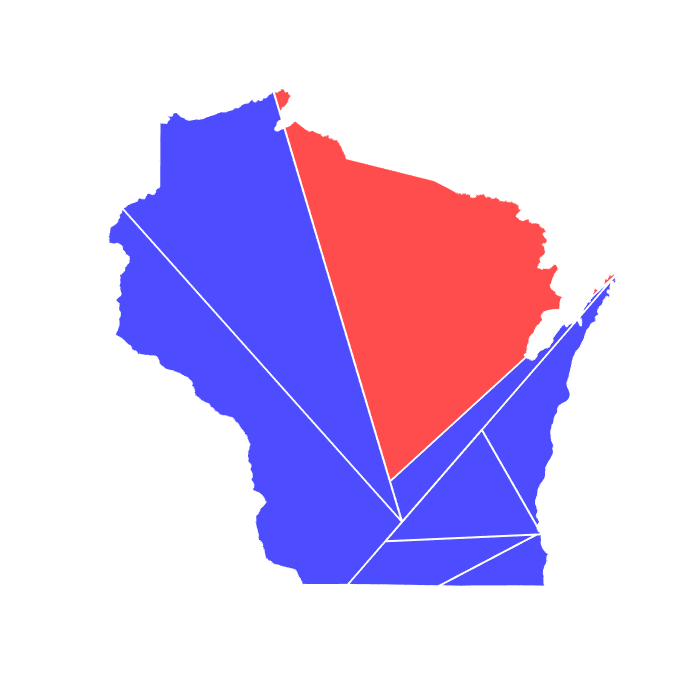}
\caption{\label{fig.example}
\footnotesize{Partitions of Wisconsin into 8 notional voting districts.
Each partition was obtained by selecting lines (``split-lines'') to iteratively split regions into two nearly equal populations.
For the left-hand partition, these lines were selected so as to maximize the resulting number of majority-Republican districts, whereas majority-Democrat districts were encouraged for the partition on the right.
Here, we applied the 2016 presidential election returns~\cite{Wisconsin:data} as a proxy for the spatial distribution of Republicans and Democrats.
Wisconsin was particularly competitive in this election, with Trump and Clinton receiving 1,405,284 and 1,382,536 votes, respectively.
Under this proxy, Republicans can win all 8 of the districts using split-lines, while Democrats can achieve 7 out of 8 (the most possible since they lost overall).
Our main result demonstrates that for a certain model of voter locations and preferences, one may use split-lines to gerrymander a competitive state and win an average of over 70\% of the allotted districts.}}
\end{figure}

\section{The model and main result}

In order to analyze the effectiveness of partisan gerrymandering with geographically compact districts, we need a model for voter locations and preferences.
For this, we introduce the \textbf{$(k,n)$-stochastic voter circle model}, which takes $kn$ voters at distinct locations of a circle who cast independent random votes uniformly from $\{\pm1\}$.
(Here, ``circle'' refers to the 1-dimensional boundary of a disk.)
Given this election data, a partisan mapmaker then partitions the circle into $k$ districts so as to maximize the number of majority-positive districts.
In this model, we enforce an extreme version of ``one person, one vote''~\cite{Smith:14} by requiring each district to contain exactly $n$ voters, and we enforce ``geographic compactness'' by requiring each district to be a contiguous portion of the circle.
We refer to any such partition that maximizes the number of majority-positive districts as an \textbf{optimal partisan gerrymander}.

We will devote our attention to the special case where $k=2$.
This case brings a nice interpretation in which voter locations enjoy a more general configuration in the plane:
Take $n$ voter locations that, together with their centroid, are in general position, and reflect these locations about the centroid to obtain the other $n$ voter locations.
One may project these $2n$ locations onto the unit circle centered at the centroid, and identify permissible partitions of the circle with partitions of the plane into two convex districts that evenly divide the voters.
Such partitions of the plane are arguably the most geographically compact possible:

\begin{theorem}[Optimal geographic compactness]
Partition a closed disk $D$ of unit radius into two regions $A,B$ whose closures are homeomorphic to $D$.
Then
\[
\max\Big\{|\partial A|,|\partial B|\Big\}\geq \pi+2,
\qquad
\min\bigg\{\frac{|A|}{|\operatorname{hull}(A)|},\frac{|B|}{|\operatorname{hull}(B)|}\bigg\}\leq 1,
\qquad
I_A+I_B\geq\frac{\pi}{2}-\frac{16}{9\pi}.
\]
Equality is simultaneously achieved in all three when $A$ and $B$ are complementary half-disks.
\end{theorem}

\begin{proof}
First, each point in $\partial D$ must lie in $\partial A \cup \partial B$.
If $\partial D\subseteq\partial A$ or $\partial D\subseteq\partial B$, then
\[
\max\Big\{|\partial A|,|\partial B|\Big\}
\geq|\partial D|
=2\pi
\geq \pi+2.
\]
Otherwise, $\partial A\cap\partial D$ and $\partial B\cap\partial D$ are both nonempty, and we claim they are both connected.

To see this, suppose to the contrary that there exist $a_1,a_2\in\partial A\cap\partial D$ and $b_1,b_2\in\partial B\cap\partial D$, all distinct, arranged in counter-clockwise order as $a_1$, $b_1$, $a_2$, $b_2$.
Then since the closure $\overline{A}$ is homeomorphic to $D$, there exists a path $P$ from $a_1$ to $a_2$ whose interior points lie in the interior $A^\circ$.
One may also draw a path from $a_2$ to $a_1$ by extending radially to the concentric circle of radius $2$, and then orbiting along this circle before descending to $a_1$.
Combined, these two paths produce a simple closed curve $C$ that separates $b_1$ from $b_2$.
Since $\overline{B}$ is homeomorphic to $D$, there exists a path $Q$ from $b_1$ to $b_2$ whose interior points lie in $B^\circ$.
The Jordan curve theorem then gives that an interior point $q$ of $Q$ lies in $C$, and furthermore, $q\in B^\circ\subseteq D$ implies that $q\in C\cap D$, i.e., $q$ is also an interior point of $P$.
Overall, $q\in A^\circ\cap B^\circ$, violating the assumption that $A$ and $B$ are disjoint.

At this point, we know that $C_A=\partial A\cap\partial D$ and $C_B=\partial B\cap\partial D$ are nonempty and connected.
Let $x,y$ denote the endpoints of $C_A$ (also, of $C_B$).
Since $\overline{A},\overline{B}$ are homeomorphic to $D$, there exists a path $R$ from $x$ to $y$ such that $\partial A=C_A\cup R$ and $\partial B=C_B\cup R$.
As such,
\[
\max\Big\{|\partial A|,|\partial B|\Big\}
\geq\max\Big\{|C_A|,|C_B|\Big\}+|R|
\geq\max\Big\{|C_A|,2\pi-|C_A|\Big\}+\|x-y\|.
\]
Without loss of generality, we may take $x=(\cos\theta,\sin\theta)$ and $y=(\cos\theta,-\sin\theta)$, and so
\[
\max\Big\{|\partial A|,|\partial B|\Big\}
\geq\max\Big\{|C_A|,2\pi-|C_A|\Big\}+\|x-y\|
=\max\Big\{2\theta,2\pi-2\theta\Big\}+2\sin\theta.
\]
The derivative of the right-hand side is negative for $\theta\in(0,\pi/2)$ and positive for $\theta\in(\pi/2,\pi)$, meaning the right-hand side is minimized by $\theta=\pi/2$, thereby producing the desired bound.
In addition, the perimeter of a half-disk is $\pi+2$, achieving equality in this bound.

For the second bound, we note that every $A\subseteq D$ satisfies $A\subseteq\operatorname{hull}(A)$, and so $|A|\leq|\operatorname{hull}(A)|$ with equality when $A$ is convex, for example, when $A$ is a half-disk.

For the last bound, let $\mu_A$ and $\mu_B$ denote the centroids of $A$ and $B$.
If $\mu_A=\mu_B$, then
\[
I_A+I_B
=\iint_D\|x-\mu_A\|^2d\operatorname{Leb}(x)
\geq\iint_D\|x\|^2d\operatorname{Leb}(x)
=\frac{\pi}{2}
\geq\frac{\pi}{2}-\frac{16}{9\pi},
\]
where the first inequality follows from the fact that the moment of inertia is minimal about the center of mass.
Now suppose $\mu_A\neq\mu_B$.
Then we may define $A'$ and $B'$ as an alternative partition of $D$ obtained from the perpendicular bisector of $\mu_A$ and $\mu_B$.
($A'$ and $B'$ are known as the Voronoi regions of $\mu_A$ and $\mu_B$, respectively.)
Let $\mu_{A'}$ and $\mu_{B'}$ denote the centroids of $A'$ and $B'$.
Then
\begin{align*}
I_A+I_B
&=\iint_D\Big(\mathbf{1}_A(x)\|x-\mu_A\|^2+\mathbf{1}_B(x)\|x-\mu_B\|^2\Big)d\operatorname{Leb}(x)\\
&\geq\iint_D\Big(\mathbf{1}_{A'}(x)\|x-\mu_A\|^2+\mathbf{1}_{B'}(x)\|x-\mu_B\|^2\Big)d\operatorname{Leb}(x)\\
&\geq\iint_D\Big(\mathbf{1}_{A'}(x)\|x-\mu_{A'}\|^2+\mathbf{1}_{B'}(x)\|x-\mu_{B'}\|^2\Big)d\operatorname{Leb}(x)
=I_{A'}+I_{B'},
\end{align*}
where the inequalities follow from comparing integrands pointwise.
As such, we may restrict our attention to partitions of $D$ that arise from a separating line.
Without loss of generality, $A$ and $B$ are of the form
\[
A=\Big\{(x,y)\in D:x\leq z\Big\},
\qquad
B=\Big\{(x,y)\in D: x>z\Big\}
\]
for some $z\in(-1,1)$.
Then $I_A+I_B$ is a function of $z$ that is minimized at $z=0$ (see Lemma~\ref{lem.moment of inertia} for details), in which case a bit of calculus gives $I_A+I_B=\frac{\pi}{2}-\frac{16}{9\pi}$.
\end{proof}

\begin{theorem}[Main result]
For $n\geq1$, let $D_n$ denote the random number of majority-positive districts in an optimal partisan gerrymander under the $(2,n)$-stochastic voter circle model.
Then
\[
\operatorname{Pr}(D_n=2)=\frac{1}{2}-\Theta\Big(\frac{1}{\sqrt{n}}\Big),
\qquad
\lim_{n\rightarrow\infty}\operatorname{Pr}(D_n=0)=\frac{1}{1+e^\pi}.
\]
\end{theorem}

What follows is a proof of the $D_n=2$ claim based on a discrete version of the intermediate value theorem.
Let $P$ denote the total number of positive votes.
We claim that $D_n=2$ precisely when $P\geq 2\lfloor n/2+1\rfloor$, that is, $P\geq n+1$ for $n$ odd and $P\geq n+2$ for $n$ even.
It suffices to prove this claim since
\[
\operatorname{Pr}(P\geq n+1)
=\frac{1}{2}-\frac{1}{2^{2n+1}}\binom{2n}{n},
\qquad
\operatorname{Pr}(P\geq n+2)
=\frac{1}{2}-\frac{1}{2^{2n+1}}\binom{2n}{n}-\frac{1}{2^{2n}}\binom{2n}{n+1},
\]
both of which equal $1/2-\Theta(1/\sqrt{n})$ by Stirling's approximation.
Since $\lfloor n/2+1\rfloor$ positive votes are required to carry a size-$n$ district, $P<2\lfloor n/2+1\rfloor$ implies $D_n<2$.
Now suppose $P\geq 2\lfloor n/2+1\rfloor$.
It is convenient to index the voter locations by $\mathbb{Z}/2n\mathbb{Z}$.
Let $P_i$ denote the random number of positive votes in $[i,i+n)$.
If $|P_i-P_{i+n}|\leq 1$ for some $i\in\mathbb{Z}/2n\mathbb{Z}$, then
\[
\min\{P_i,P_{i+n}\}
\geq\frac{P_i+P_{i+n}}{2}-\frac{1}{2}
=\frac{P}{2}-\frac{1}{2}
\geq \Big\lfloor\frac{n}{2}+1\Big\rfloor-\frac{1}{2}
>\Big\lfloor\frac{n}{2}\Big\rfloor,
\]
and so $D_n=2$.
As such, we are done if $|P_0-P_n|\leq1$.
Otherwise, $|P_0-P_n|\geq2$, and we may take $P_0-P_n\geq2$ without loss of generality.
Put $s_i:=P_i-P_{i+n}$ for $i\in\{0,\ldots,n\}$.
We have $s_0\geq2$, $s_n\leq -2$, and $s_{i+1}-s_i\in\{-2,0,2\}$ for each $i\in\{0,\ldots,n\}$.
Let $j$ be the smallest such that $s_j<2$.
Then $2>s_j=s_{j-1}-2\geq0$, i.e., $|P_j-P_{j+n}|=|s_j|\leq1$, and so we are done.

Our proof of the $D_n=0$ claim is longer, and can be found in the following section.
Note that we analyzed $D_n=2$ for every fixed $n$ and then applied Stirling's approximation to deduce the reported asymptotic behavior.
To analyze our $D_n=0$ claim, we relate the stochastic voter circle model to a random walk.
When suitably scaled, Donsker's invariance principle gives that the walk's distribution converges to that of standard Brownian motion, and so the desired limit can be written as the probability that standard Brownian motion lies in some event.
To compute this probability, we leverage facts about Brownian bridges that have no analog when working with random walks.
As such, it is not obvious how to modify our approach to analyze $D_n=0$ for a fixed $n$.
This is why we report a convergence rate for $\operatorname{Pr}(D_n=2)$ but not for $\operatorname{Pr}(D_n=0)$.

\section{Proof of the main result}

Draw independent votes $\{x_i\}_{i\in\mathbb{Z}/2n\mathbb{Z}}$ uniformly from $\{\pm1\}$.
We seek the probability that the sum over every interval of length $n$ is nonpositive.
Define the random walks
\[
A_i=\sum_{j=0}^{i-1} x_j,
\qquad
B_i=\sum_{j=0}^{i-1} x_{j+n},
\qquad
i\in\{0,\ldots,n\}.
\]
Observe that $A_0=B_0=0$.
Then the sums of votes over districts $[i,i+n)$ and $[i+n,i)$ are
\[
\sum_{j=i}^{i+n-1}x_j
=A_n-A_i+B_i,
\qquad
\sum_{j=i+n}^{i-1}x_j
=B_n-B_i+A_i,
\]
respectively.
As such, both sums are nonpositive for every $i\in\{0,\ldots,n\}$ precisely when
\[
A_n\leq A_i-B_i \leq -B_n
\qquad \forall i\in\{0,\ldots,n\}.
\]
Define rescaled random walks by
\[
W^{(n)}_1(t)=\frac{A_{\lfloor nt\rfloor}}{\sqrt{n}},
\qquad
W^{(n)}_2(t)=-\frac{B_{\lfloor nt\rfloor}}{\sqrt{n}},
\qquad
t\in[0,1],
\]
and let $W_1$ and $W_2$ denote independent instances of standard Brownian motion.
Then
\begin{align*}
&\lim_{n\rightarrow\infty}\operatorname{Pr}(D_n=0)\\
&=\lim_{n\rightarrow\infty}\operatorname{Pr}\Big(A_n\leq \min_{i\in\{0,\ldots,n\}}(A_i-B_i)\leq \max_{i\in\{0,\ldots,n\}}(A_i-B_i)\leq -B_n\Big)\\
&=\lim_{n\rightarrow\infty}\operatorname{Pr}\Big(W_1^{(n)}(1)\leq \min_{t\in[0,1]}(W_1^{(n)}(t)+W_2^{(n)}(t))\leq \max_{t\in[0,1]}(W_1^{(n)}(t)+W_2^{(n)}(t))\leq W_2^{(n)}(1)\Big)\\
&=\operatorname{Pr}\Big(W_1(1)\leq \min_{t\in[0,1]}(W_1(t)+W_2(t))\leq \max_{t\in[0,1]}(W_1(t)+W_2(t))\leq W_2(1)\Big),
\end{align*}
where the last step follows from Donsker's invariance principle (see the proof of Lemma~\ref{lem.donsker} for details).
Next, consider Brownian bridges $B_k(t)=W_k(t)-tW_k(1)$.
Then $(W_1(1),B_1(\cdot),W_2(1),B_2(\cdot))$ are independent, with $W_1(1)$ and $W_2(1)$ exhibiting standard normal distribution.
Put
\[
X=\frac{W_1(1)+W_2(1)}{\sqrt{2}},
\qquad
Z=\frac{W_1(1)-W_2(1)}{\sqrt{2}}.
\]
Then $X$ and $Z$ have standard normal distribution with $(X,Z,B_1(\cdot),B_2(\cdot))$ independent.
As such,
\[
W(t)=\frac{B_1(t)+B_2(t)}{\sqrt{2}}+tX,
\qquad
t\in[0,1]
\]
is an instance of standard Brownian motion that is independent of $Z$.
Put $L=\min_{t\in[0,1]}W(t)$ and $U=\max_{t\in[0,1]}W(t)$.
We continue our calculation:
\begin{align*}
\lim_{n\rightarrow\infty}\operatorname{Pr}(D_n=0)
&=\operatorname{Pr}\Big(W_1(1)\leq \min_{t\in[0,1]}(W_1(t)+W_2(t))\leq \max_{t\in[0,1]}(W_1(t)+W_2(t))\leq W_2(1)\Big)\\
&=\operatorname{Pr}\bigg(\frac{X+Z}{2}\leq L\leq U\leq \frac{X-Z}{2}\bigg)\\
&=\int_{-\infty}^{\infty}\operatorname{Pr}\bigg(\frac{X+z}{2}\leq L\leq U\leq \frac{X-z}{2}\bigg)\cdot\frac{1}{\sqrt{2\pi}}e^{-z^2/2}dz,
\end{align*}
where the last step conditions on $Z$.
Observe the main result in~\cite{ChoiR:13}:
\begin{equation*}
\label{eq.trivariate1}
\operatorname{Pr}\Big(a\leq L\leq U\leq b \wedge X\in dx\Big)
=f(a,b,x)dx,
\end{equation*}
where
\begin{equation*}
\label{eq.trivariate2}
f(a,b,x)=\left\{\begin{array}{ll}\displaystyle\frac{1}{\sqrt{2\pi}}\sum_{k=-\infty}^\infty\Big(e^{-(x-2k(b-a))^2/2}-e^{-(x-2b-2k(b-a))^2/2}\Big)&\text{if }a\leq 0\leq b\\
0&\text{otherwise.}
\end{array}
\right.
\end{equation*}
Since $f$ is continuous~\cite{ChoiR:13}, Lemma~\ref{lem.how to use trivariate} allows us to leverage this expression:
\begin{align*}
\lim_{n\rightarrow\infty}\operatorname{Pr}(D_n=0)
&=\int_{-\infty}^{\infty}\operatorname{Pr}\bigg(\frac{X+z}{2}\leq L\leq U\leq \frac{X-z}{2}\bigg)\cdot\frac{1}{\sqrt{2\pi}}e^{-z^2/2}dz\\
&=\int_{-\infty}^\infty\Bigg[\int_{-\infty}^\infty f\Big(\frac{x+z}{2},\frac{x-z}{2},x\Big)dx\Bigg]\cdot\frac{1}{\sqrt{2\pi}}e^{-z^2/2}dz\\
&=\int_{0}^\infty\int_{-\infty}^0 f(a,b,a+b)\cdot\frac{1}{\sqrt{2\pi}}e^{-(a-b)^2/2}\cdot 2dadb.
\end{align*}
We wish to convert the integral of the series to a series of integrals.
To this end, define
\[
I_k=\int_0^\infty\int_{-\infty}^0e^{-[(a+b-2k(b-a))^2+(a-b)^2]/2}dadb,
\quad
J_k=\int_0^\infty\int_{-\infty}^0e^{-[(a-b-2k(b-a))^2+(a-b)^2]/2}dadb.
\]
Then the triangle inequality gives
\[
\sum_{k=-\infty}^\infty\int_0^\infty\int_{-\infty}^0\Big|e^{-[(a+b-2k(b-a))^2+(a-b)^2]/2}-e^{-[(a-b-2k(b-a))^2+(a-b)^2]/2}\Big|dadb
\leq\sum_{k=-\infty}^\infty I_k+\sum_{k=-\infty}^\infty J_k,
\]
which is finite by Lemmas~\ref{lem.Ik} and~\ref{lem.Jk}.
As such, the Fubini--Tonelli theorem allows us to continue:
\begin{align*}
\lim_{n\rightarrow\infty}\operatorname{Pr}(D_n=0)
&=\int_{0}^\infty\int_{-\infty}^0 f(a,b,a+b)\cdot\frac{1}{\sqrt{2\pi}}e^{-(a-b)^2/2}\cdot 2dadb\\
&=\frac{1}{2\pi}\sum_{k=-\infty}^\infty\int_0^\infty\int_{-\infty}^0\Big(e^{-[(a+b-2k(b-a))^2+(a-b)^2]/2}-e^{-[(a-b-2k(b-a))^2+(a-b)^2]/2}\Big)dadb\\
&=\frac{1}{2\pi}\bigg(\sum_{k=-\infty}^\infty I_k-\sum_{k=-\infty}^\infty J_k\bigg)\\
&=\frac{1}{1+e^\pi},
\end{align*}
where the last step follows from Lemmas~\ref{lem.Ik} and~\ref{lem.Jk}.

\section{Lemmata}

\begin{lemma}
\label{lem.center of mass}
Given a continuous function $\rho\colon[a,b]\rightarrow[0,\infty)$, denote
\[
M=\int_a^b\rho(x)dx,
\qquad
\overline{x}=\frac{1}{M}\int_a^bx\rho(x) dx.
\]
\begin{itemize}
\item[(a)]
If $\rho(\frac{a+b}{2}+t)\geq\rho(\frac{a+b}{2}-t)$ for every $t\in[0,\frac{b-a}{2}]$, then $\overline{x}\geq\frac{a+b}{2}$.
\item[(b)]
If $\rho(\frac{a+b}{2}+t)\leq\rho(\frac{a+b}{2}-t)$ for every $t\in[0,\frac{b-a}{2}]$, then $\overline{x}\leq\frac{a+b}{2}$.
\end{itemize}
\end{lemma}

\begin{proof}
We will prove (a), and (b) follows by negating $x$.
Let $\rho_1$ denote the symmetric part of $\rho$:
\[
\rho_1(x)=\left\{\begin{array}{ll}\rho(x)&\text{if }x\leq\frac{a+b}{2}\\\rho(a+b-x)&\text{otherwise,}\end{array}\right.
\]
put $\rho_2(x)=\rho(x)-\rho_1(x)$, and denote
\[
M_1=\int_a^b\rho_1(x)dx,
\qquad
M_2=\int_a^b\rho_2(x)dx.
\]
Then the symmetry of $\rho_1$ and the fact that $\rho_2(x)=0$ for $x\leq\frac{a+b}{2}$ together give
\begin{align}
\overline{x}
=\frac{1}{M}\int_a^bx\rho(x) dx
\nonumber&=\frac{M_1}{M}\bigg(\frac{1}{M_1}\int_a^bx\rho_1(x) dx\bigg)+\frac{M_2}{M}\bigg(\frac{1}{M_2}\int_a^bx\rho_2(x) dx\bigg)\\
\label{eq.center of mass bound}&=\frac{M_1}{M}\bigg(\frac{a+b}{2}\bigg)+\frac{M_2}{M}\bigg(\frac{1}{M_2}\int_{\frac{a+b}{2}}^bx\rho_2(x) dx\bigg).
\end{align}
Finally, 
\[
\frac{1}{M_2}\int_{\frac{a+b}{2}}^bx\rho_2(x) dx
\geq\frac{1}{M_2}\int_{\frac{a+b}{2}}^b\bigg(\frac{a+b}{2}\bigg)\rho_2(x) dx
=\frac{a+b}{2}\cdot\frac{1}{M_2}\int_{a}^b\rho_2(x) dx
=\frac{a+b}{2}.
\]
Combining with \eqref{eq.center of mass bound} and observing $M=M_1+M_2$ then gives the result.
\end{proof}

\begin{lemma}
\label{lem.moment of inertia}
Consider the following quantities, defined for $z\in(-1,1]$:
\[
M(z)=\int_{-1}^z 2\sqrt{1-x^2}dx,
\qquad
\overline{x}(z)=\frac{1}{M(z)}\int_{-1}^z x\cdot 2\sqrt{1-x^2}dx,
\]
\[
I(z)=\int_{-1}^z\int_{-\sqrt{1-x^2}}^{\sqrt{1-x^2}}\Big(y^2+\big(x-\overline{x}(z)\big)^2\Big)dydx.
\]
Then $I(z)+I(-z)\geq 2I(0)$ for all $z\in(-1,1)$.
\end{lemma}

\begin{proof}
Let $D$ denote the unit disk in the $xy$-plane, and let $R_z$ denote the portion of this disk satisfying $x\leq z$.
We start by finding a more convenient expression for the function we are minimizing:
\begin{align*}
I(z)+I(-z)
&=\iint_{R_z}\Big(y^2+\big(x-\overline{x}(z)\big)^2\Big)dxdy+\iint_{R_{-z}}\Big(y^2+\big(x-\overline{x}(-z)\big)^2\Big)dxdy\\
&=\iint_D(x^2+y^2)dxdy
-2\overline{x}(z)\iint_{R_z}xdxdy+\overline{x}(z)^2M(z)\\
&\qquad\qquad\qquad\qquad-2\overline{x}(-z)\iint_{R_{-z}}xdxdy+\overline{x}(-z)^2M(-z)\\
&=\frac{\pi}{2}-\overline{x}(z)^2M(z)-\overline{x}(-z)^2M(-z)
\end{align*}
Next, the fundamental theorem of calculus gives
\[
\overline{x}(z)=-\frac{2}{3}\cdot\frac{(1-z^2)^{3/2}}{M(z)},
\qquad
\frac{d}{dz}M(z)=2\sqrt{1-z^2}.
\]
These identities allow us to simplify the derivative of our function:
\begin{equation}
\label{eq.derivative}
\frac{d}{dz}\Big(I(z)+I(-z)\Big)
=2\sqrt{1-z^2}\Big(\overline{x}(z)+\overline{x}(-z)\Big)\Big(\overline{x}(z)-\overline{x}(-z)-2z\Big).
\end{equation}
By symmetry, it suffices to show that \eqref{eq.derivative} is nonnegative for $z\in(0,1)$.
To this end, applying Lemma~\ref{lem.center of mass}(b) to
\[
\rho(x)=\left\{\begin{array}{ll}
2\sqrt{1-x^2}&\text{if }x\in[-1,z]\\
0&\text{if }x\in(z,1]
\end{array}\right.
\]
gives that $\overline{x}(z)\leq0$ for every $z\in(-1,1)$.
It remains to show that 
\[
\Big(\overline{x}(z)-z\Big)-\Big(\overline{x}(-z)-(-z)\Big)
=\overline{x}(z)-\overline{x}(-z)-2z
\leq0
\]
for $z\in(0,1)$, and it suffices to show $\overline{x}(z)-z$ is decreasing over $z\in(-1,1)$.
To this end, it is straightforward to compute the derivative:
\[
\frac{d}{dz}\Big(\overline{x}(z)-z\Big)
=3\cdot\frac{\overline{x}(z)^2-z\overline{x}(z)}{1-z^2}-1.
\]
To estimate this quantity, we apply Lemma~\ref{lem.center of mass}(a) to $\rho(x)=2\sqrt{1-x^2}$ over $x\in[-1,z]$ to get $\overline{x}(z)\geq\frac{z-1}{2}$.
Combining with the more trivial bounds $\overline{x}(z)\leq z\leq 1$ then gives
\[
-\frac{1}{2}
\leq\overline{x}(z)-\frac{z}{2}
\leq \frac{z}{2}
\leq\frac{1}{2}.
\]
As such, we complete the square to get
\[
\frac{d}{dz}\Big(\overline{x}(z)-z\Big)
=3\cdot\frac{(\overline{x}-\frac{z}{2})^2-\frac{z^2}{4}}{1-z^2}-1
\leq3\cdot\frac{(\frac{1}{2})^2-\frac{z^2}{4}}{1-z^2}-1
=-\frac{1}{4}
<0,
\]
as desired.
\end{proof}

\begin{lemma}
\label{lem.donsker}
Consider the random variables
\begin{align*}
m^{(n)}&=\min_{t\in[0,1]}(W_1^{(n)}(t)+W_2^{(n)}(t)),
&m=\min_{t\in[0,1]}(W_1(t)+W_2(t)),\\
M^{(n)}&=\max_{t\in[0,1]}(W_1^{(n)}(t)+W_2^{(n)}(t)),
&M=\max_{t\in[0,1]}(W_1(t)+W_2(t)).
\end{align*}
Then $\displaystyle\lim_{n\rightarrow\infty}\operatorname{Pr}\Big(W_1^{(n)}(1)\leq m^{(n)}\leq M^{(n)}\leq W_2^{(n)}(1)\Big)=\operatorname{Pr}\Big(W_1(1)\leq m\leq M\leq W_2(1)\Big)$.
\end{lemma}

\begin{proof}
Consider the stochastic processes
\[
V^{(n)}(t)
=\left\{\begin{array}{ll}
W_1^{(n)}(t)&\text{if }t\in[0,1]\\
W_1^{(n)}(1)+W_2^{(n)}(t-1)&\text{if }t\in(1,2],
\end{array}\right.
\quad
V(t)
=\left\{\begin{array}{ll}
W_1(t)&\text{if }t\in[0,1]\\
W_1(1)+W_2(t-1)&\text{if }t\in(1,2].
\end{array}\right.
\]
Donsker's invariance principle gives that $V^{(n)}(\cdot)$ converges in distribution to the standard Brownian motion $V(\cdot)$.
From this, we may conclude that $(W_1^{(n)}(1),m^{(n)},M^{(n)},W_2^{(n)}(1))$ converges in distribution to $(W_1(1),m,M,W_2(1))$.
To obtain the desired result, it suffices to show that $S=\{(a,b,c,d):a\leq b\leq c\leq d\}$ is a continuity set of $(W_1(1),m,M,W_2(1))$, that is, that
\[
\operatorname{Pr}\Big((W_1(1),m,M,W_2(1))\in \partial S\Big)=0,
\]
where $\partial S$ denotes the boundary of $S$ in the standard topology of $\mathbb{R}^4$.
Observe that
\[
\partial S
=\Big\{(a,b,c,d):a=b ~\vee~ b=c ~\vee~ c=d\Big\},
\]
and so the union bound gives
\[
\operatorname{Pr}\Big((W_1(1),m,M,W_2(1))\in \partial S\Big)
\leq\operatorname{Pr}\Big(W_1(1)=m\Big)+\operatorname{Pr}\Big(m=M\Big)+\operatorname{Pr}\Big(M=W_2(1)\Big).
\]
First, $(W_1(\cdot)+W_2(\cdot))/\sqrt{2}$ is standard Brownian motion, which almost surely takes strictly positive values and strictly negative values on $(0,1)$, and so $m<M$ with probability one.
Below, we demonstrate $\operatorname{Pr}(W_1(1)=m)=0$, and a similar argument gives $\operatorname{Pr}(M=W_2(1))=0$, thereby producing the result.

Consider Brownian bridges $B_k(t)=W_k(t)-tW_k(1)$ and observe that $(W_1(1),B_1(\cdot),W_2(1),B_2(\cdot))$ are independent, with $W_1(1)$ and $W_2(1)$ exhibiting standard normal distribution.
Put
\[
X=\frac{W_1(1)+W_2(1)}{\sqrt{2}},
\qquad
Z=\frac{W_1(1)-W_2(1)}{\sqrt{2}}.
\]
Then $X$ and $Z$ have standard normal distribution with $(X,Z,B_1(\cdot),B_2(\cdot))$ independent.
In particular, $Z$ is independent of $W_1(\cdot)+W_2(\cdot)$.
Conditioning on $W_1(\cdot)+W_2(\cdot)$ then gives
\begin{align*}
\operatorname{Pr}\Big(W_1(1)=m\Big)
&=\mathbb{E}\operatorname{Pr}\Big(W_1(1)=m~\Big|~W_1(\cdot)+W_2(\cdot)\Big)\\
&=\mathbb{E}\operatorname{Pr}\Big(X+Z=\sqrt{2}m~\Big|~W_1(\cdot)+W_2(\cdot)\Big)
=0,
\end{align*}
where the last step follows from the fact that $W_1(\cdot)+W_2(\cdot)$ determines $X$ and $m$, while the distribution of $Z$ is continuous.
\end{proof}

\begin{lemma}
\label{lem.how to use trivariate}
Let $L$, $U$ and $X$ be random variables for which there exists a continuous function $f$ such that
\[
\operatorname{Pr}\Big(a\leq L\leq U\leq b \wedge X\in[c,d)\Big)
=\int_c^d f(a,b,x)dx
\qquad
\forall a,b,c,d\in\mathbb{R},~a<b,~c<d.
\]
Then for any continuous functions $a$ and $b$ with $a<b$ pointwise, we have
\[
\operatorname{Pr}\Big(a(X)\leq L\leq U\leq b(X) \wedge X\in[c,d)\Big)
=\int_c^d f(a(x),b(x),x)dx
\qquad
\forall c,d\in\mathbb{R},~c<d.
\]
\end{lemma}

\begin{proof}
For each $h$, consider all possible partitions of $[c,d)$ into finitely many half-open intervals $\{I_i\}$ of length less than $h$.
In each interval $I_i$, consider all possible choices of $x_i^*,x_i^{**}\in \overline{I_i}$.
We claim that the following string of equalities hold:
\begin{align}
\lim_{h\rightarrow0}\sum_i\operatorname{Pr}\Big(a(x_i^*)\leq L\leq U\leq b(x_i^{**}) \wedge X\in I_i\Big)
\label{3}&=\lim_{h\rightarrow0}\sum_i\int_{I_i}f(a(x_i^*),b(x_i^{**}),x)dx\\
\label{4}&=\lim_{h\rightarrow0}\sum_i\int_{I_i}f(a(x),b(x),x)dx\\
\label{5}&=\int_c^d h(a(x),b(x),x)dx.
\end{align}
Indeed, \eqref{3} holds by assumption, \eqref{4} follows from the uniform continuity of $(s,t,x)\mapsto f(a(s),b(t),x)$ over $s,t,x\in[c,d]$, and \eqref{5} follows from the fact that the $I_i$'s partition $[c,d)$.
Next, define
\[
x_i^{*(1)}:=\arg\max_{x\in\overline{I_i}}a(x),
\quad
x_i^{**(1)}:=\arg\min_{x\in\overline{I_i}}b(x),
\quad
x_i^{*(2)}:=\arg\min_{x\in\overline{I_i}}a(x),
\quad
x_i^{**(2)}:=\arg\max_{x\in\overline{I_i}}b(x).
\]
We obtain the following estimates:
\begin{align}
\operatorname{Pr}\Big(a(x_i^{*(1)})\leq L\leq U\leq b(x_i^{**(1)})\wedge X\in I_i\Big)
\nonumber&\leq \operatorname{Pr}\Big(a(X)\leq L\leq U\leq b(X) \wedge X\in I_i\Big)\\
\label{squeeze}&\leq \operatorname{Pr}\Big(a(x_i^{*(2)})\leq L\leq U\leq b(x_i^{**(2)})\wedge X\in I_i\Big).
\end{align}
At this point, we claim
\begin{align}
\operatorname{Pr}\Big(a(X)\leq L\leq U\leq b(X) \wedge X\in[c,d)\Big)
\label{1}&=\lim_{h\rightarrow0}\sum_i\operatorname{Pr}\Big(a(X)\leq L\leq U\leq b(X) \wedge X\in I_i\Big)\\
\label{2}&=\lim_{h\rightarrow0}\sum_i\operatorname{Pr}\Big(a(x_i^*)\leq L\leq U\leq b(x_i^{**}) \wedge X\in I_i\Big)\\
\label{end}&=\int_c^d f(a(x),b(x),x)dx.
\end{align}
Indeed, \eqref{1} follows from equality in the union bound: the $I_i$'s partition $[c,d)$ so that the events in the union are disjoint.
Finally, \eqref{2} applies the squeeze theorem to \eqref{squeeze}, and \eqref{end} comes from \eqref{5}.
\end{proof}

\begin{lemma}
\label{lem.Ik}
$\displaystyle\sum_{k=-\infty}^\infty I_k=\frac{\pi}{2}.$
\end{lemma}

\begin{proof}
Recall $I_k=\int_0^\infty\int_{-\infty}^0e^{-[(a+b-2k(b-a))^2+(a-b)^2]/2}dadb$, and diagonalize the exponent:
\[
\frac{1}{2}\Big[\Big(a+b-2k(b-a)\Big)^2+(a-b)^2\Big]
=\Big(-ka+(k-1)b\Big)^2+\Big(-(k+1)a+kb\Big)^2.
\]
This motivates the change of variables $(a,b)\mapsto(x,y)$ given by
\[
\left[\begin{array}{c}x\\y\end{array}\right]
=\left[\begin{array}{cc}-k&k-1\\-(k+1)&k\end{array}\right]\left[\begin{array}{c}a\\b\end{array}\right].
\]
Observe that the above matrix has determinant $-1$ (as does its inverse).
As such, we have
\[
I_k=\iint_{R_k}e^{-(x^2+y^2)}dxdy,
\]
where $R_k$ denotes the cone generated by $(k,k+1)$ and $(k-1,k)$.
Notice that
\[
\bigcup_{k=-\infty}^\infty R_k
=\big\{(x,y):x<y\big\}\cup\big\{(0,0)\big\},
\]
and furthermore, $R_i\cap R_j$ has measure zero whenever $i\neq j$.
As such, the desired series combines integrals into a computable one:
\[
\sum_{k=-\infty}^\infty I_k
=\int_{-\infty}^\infty\int_{-\infty}^y e^{-(x^2+y^2)}dxdy
=\int_{\pi/4}^{5\pi/4}\int_0^\infty e^{-r^2}rdrd\theta
=\frac{\pi}{2}.
\qedhere
\]
\end{proof}

\begin{lemma}
\label{lem.Jk}
$\displaystyle\sum_{k=-\infty}^\infty J_k=\frac{\pi}{2}\tanh\Big(\frac{\pi}{2}\Big).$
\end{lemma}

\begin{proof}
Recall $J_k=\int_0^\infty\int_{-\infty}^0e^{-[(a-b-2k(b-a))^2+(a-b)^2]/2}dadb$, and simplify the exponent:
\[
\frac{1}{2}\Big[\Big(a-b-2k(b-a)\Big)^2+(a-b)^2\Big]
=(2k^2+2k+1)(-a+b)^2.
\]
This motivates the change of variables $x=-a+b$, $y=-a-b$:
\[
J_k
=\int_0^\infty\int_{-x}^x e^{-(2k^2+2k+1)x^2}\cdot\frac{1}{2}dydx
=\frac{1}{4k^2+2k+2}.
\]
At this point, a partial fractions decomposition gives
\begin{align*}
\sum_{k=-\infty}^\infty J_k
&=\sum_{k=-\infty}^\infty\frac{1}{4k^2+2k+2}\\
&=2\sum_{k=0}^\infty\frac{1}{4k^2+2k+2}\\
&=2\sum_{k=0}^\infty\bigg(\frac{i/4}{k+\frac{1+i}{2}}-\frac{i/4}{k+\frac{1-i}{2}}\bigg)\\
&=\frac{i}{2}\Bigg\{\bigg[-\gamma+\sum_{k=0}^\infty\bigg(\frac{1}{k+1}-\frac{1}{k+\frac{1-i}{2}}\bigg)\bigg]-\bigg[-\gamma+\sum_{k=0}^\infty\bigg(\frac{1}{k+1}-\frac{1}{k+\frac{1+i}{2}}\bigg)\bigg]\Bigg\}\\
&=\frac{i}{2}\bigg(\psi\Big(\frac{1-i}{2}\Big)-\psi\Big(\frac{1+i}{2}\Big)\bigg),
\end{align*}
where $\gamma$ denotes the Euler--Mascheroni constant, and $\psi(\cdot)$ is the digamma function defined in terms of the gamma function by $\psi(x)=\frac{d}{dx}\log(\Gamma(x))$.
Indeed, the last equality above follows from~\cite{AbramowitzS:72}.
The digamma function satisfies the following reflection formula~\cite{AbramowitzS:72}:
\[
\psi(1-z)-\psi(z)=\pi\cot\pi z,
\qquad
z\in\mathbb{C}.
\]
Taking $z=\frac{1+i}{2}$ then gives
\[
\sum_{k=-\infty}^\infty J_k
=\frac{i}{2}\bigg(\psi\Big(\frac{1-i}{2}\Big)-\psi\Big(\frac{1+i}{2}\Big)\bigg)
=\frac{i}{2}\cdot\pi\cot\Big(\pi\cdot\frac{1+i}{2}\Big)
=\frac{\pi}{2}\tanh\Big(\frac{\pi}{2}\Big).
\qedhere
\]
\end{proof}

\section*{Acknowledgments}
DGM was partially supported by AFOSR F4FGA06060J007 and AFOSR Young Investigator Research Program award F4FGA06088J001.
The views expressed in this article are those of the authors and do not reflect the official policy or position of the authors' employers, the United States Air Force, Department of Defense, or the U.S.\ Government.

\end{document}